\documentclass{iopart}
\usepackage{iopams}
\usepackage{geometry}                
\usepackage{graphicx}
\usepackage{amsthm}
\usepackage{epstopdf}
\usepackage[rgb,dvipsnames]{xcolor}

\colorlet{linkcolour}{green!40!black}
\colorlet{urlcolour}{green!40!black}

\usepackage[pdfusetitle,pagebackref]{hyperref}
\hypersetup{
	colorlinks=true,
	linkcolor=linkcolour, 
	citecolor=linkcolour,
	urlcolor=urlcolour,
	pdfauthor={Robert McLachlan, Klas Modin, Olivier Verdier},
	pdftitle={Symplectic Hopf integrator}
}

\newtheorem{theorem}{Theorem}[section]

\newtheorem{lemma}{Lemma}[section]

\theoremstyle{definition}
\newtheorem{definition}{Definition}[section]

\newtheorem{example}{Example}[section]

\def\g{\mathfrak{g}}
\def\eqalign#1{\null\,\vcenter{\openup\jot \mathsurround=0pt \ialign{\strut
     \hfil$\displaystyle{##}$&$ \displaystyle{{}##}$\hfil \crcr#1\crcr}}\,}
\def\tr{\mathop{\mathrm{tr}}}
\def\tphi{\widetilde{\varphi}}
\def\Ad{\mathrm{Ad}}

\begin{document}
\title{Collective symplectic integrators}
\author{Robert I McLachlan$^1$, Klas Modin$^2$ and Olivier Verdier$^3$}
\address{$^1$Institute of Fundamental Sciences, Massey University,
	Private Bag 11 222, Palmerston North 4442, New Zealand }
\address{$^2$Department of Mathematical Sciences, Chalmers University of Technology, Gothenburg, Sweden }
\address{$^3$Mathematics and Mathematical Statistics,
Ume\aa~Universitet, SE--901 87 Ume\aa, Sweden}
\eads{\mailto{r.mclachlan@massey.ac.nz},
\mailto{klas.modin@chalmers.se},
\mailto{olivier.verdier@gmail.com}}

\def\R{\mathbb{R}}
\def\F{{\mathcal F}}
\def\C{\mathbb{C}}
\begin{abstract}
\noindent
We construct symplectic integrators for Lie-Poisson systems. The integrators are standard symplectic (partitioned) Runge--Kutta methods. Their phase space is a symplectic vector space with a Hamiltonian action with momentum map $J$ whose range is the target Lie--Poisson manifold, and their Hamiltonian is collective, that is, it is the target Hamiltonian pulled back by $J$. The method yields, for example, a symplectic midpoint rule expressed in 4 variables for arbitrary Hamiltonians on $\mathfrak{so}(3)^*$. The method specializes in the case that a sufficiently large symmetry group acts on the fibres of $J$, and generalizes to the case that the vector space carries a bifoliation. Examples involving many classical groups are presented.
\end{abstract}


\submitto{Nonlinearity}

\section{Introduction: Symplectic integrators for canonical and noncanonical Hamiltonian systems}

A Hamiltonian system on a symplectic manifold is defined by $(M,\omega,H)$ where $M$ is a manifold, $\omega$ is a symplectic form, and $H\colon M\to\R$ is a Hamiltonian; the associated Hamiltonian vector field $X_H$ is defined by $i_{X_H} \omega = dH$, or, in local coordinates $z$ in which $\omega =\frac{1}{2} dz\wedge \Omega dz$, $\dot z = X_H(z)$ with $\Omega X_H = \nabla H$. A symplectic integrator for $X_H$ is a 1-parameter family of symplectic maps $\varphi_{\Delta t}\colon M\to M$ such that $\varphi_0 = id$ and $\left.\left(\frac{d}{dt}\varphi_{ t}\right)\right|_{t = 0} = X_H$. Symplectic integrators are known in only a few cases, the main ones being \cite{ha-lu-wa}
\begin{enumerate}
\item When $(M,\omega)$ is a symplectic vector space, many Runge--Kutta methods are symplectic. The midpoint rule $\varphi_{\Delta t}\colon z_0\mapsto z_1$, $\Omega (z_1-z_0)/\Delta t = \nabla H((z_0+z_1)/2)$ is an example.
\item When $(M,\omega$) is a symplectic vector space, then given any choice of  Darboux coordinates $(q,p)$ on $M$, many partitioned Runge--Kutta methods are symplectic. (Essentially, because $\omega = dq\wedge dp$ is now linear in $q$ and $p$, one can apply any Runge--Kutta method to the $q$ component and then determine the $p$ component by symplecticity.)
\item When $H$ can be written as $H=\sum_i H_i$ such that each split vector field $X_{H_i}$ can be integrated exactly, then compositions of their flows provide symplectic integrators.
\item When $M = T^*Q$ with its canonical symplectic form, if the configuration space $Q$ is embedded in a linear space $N$  by constraints (i.e., $Q = \{q\in N: g(q)=0\}$), then constrained symplectic integrators such as {\sc rattle} are known; the algorithm is expressed in coordinates on the linear space $T^*N$ but is designed in such a way that it induces a symplectic integrator on $M$.
\end{enumerate}

A Hamiltonian system on a Poisson manifold is defined by $(P,\{,\},H)$ where $P$ is a manifold, $\{,\}$ is a Poisson bracket, and $H\colon M\to \R$ is a Hamiltonian. Poisson maps $\varphi\colon P\to P$ are those that preserve the Poisson bracket, i.e. $\{F,G\}\circ\varphi = \{F\circ\varphi,G\circ\varphi\}$ for all $F,G\colon P\to \R$. Poisson manifolds are foliated into symplectic leaves, and often these leaves are the level sets of Casimirs, functions $C$ such that $\{C,F\}=0$ for all $F$. Poisson maps are those that preserve the foliation and pull back the symplectic form on the target leaf to the symplectic form on the source leaf. Hamiltonian vector fields are defined by their action on functions by $\dot F = \{F,H\}$. In local coordinates with $\{F,G\} = \nabla F^T K(z) \nabla G$, $X_H(z) = K(z)\nabla H(z)$. The flow of a Hamiltonian vector field is Poisson and, in addition, fixes each leaf. If there are Casimirs, then they are first integrals of $X_H$ for any $H$. Some main classes of Poisson manifolds are (i) symplectic manifolds (the case that $P$ has a single symplectic leaf); (ii) $P$ a vector space with constant Poisson tensor $K$; and (iii) $P$ a vector space with linear Poisson bracket. In this case $P$ is a Lie--Poisson manifold and we have $P=\mathfrak{g}^*$ where $\mathfrak{g}$ is a Lie algebra and $\{F,G\}(z) = z([dF,dG])$ for all $z\in\mathfrak{g}^*$. In this case the symplectic leaves of $P$ are coadjoint orbits of $G$ in $\mathfrak{g}^*$ and they form important classes of symplectic manifolds.

Symplectic integrators for $(P,\{,\},H)$ are one-parameter families of Poisson maps $\varphi_{\Delta t}:P\to P$ such that $\varphi_0 = id$ and $(\frac{d}{dt}\varphi_{t})|_{t=0} = X_H$ and, in addition, fix each leaf. These are only known in a few cases. 
\begin{enumerate}
\item When $P$ is a vector space with constant Poisson bracket, symplectic Runge--Kutta methods are Poisson integrators for any $H$ \cite{mc1}.
\item When $P=\mathfrak{g}^*$ is Lie--Poisson, many Hamiltonians (for example polynomials) can be split into integrable pieces \cite{mc-qu}.
\item When $P=\mathfrak{g}^*$ is Lie--Poisson, then $X_H$ is the Marsden--Weinstein reduction of a canonical Hamiltonian system on $T^*G$. Symplectic integrators can be constructed for any $H$ by either (i) constructing a discrete Lagrangian in $TG$ using the exponential map to provide local coordinates on $G$ \cite{ch-sc,ma-pe-sh} or (ii) embedding $G$ in a linear space and using a constrained symplectic integrator such as {\sc rattle} \cite{mc-sc}.
\end{enumerate}

However, the integrators from Case 3 are extremely complicated, involving solving implicit equations in Lie groups,  infinite series of Lie brackets, and/or using an excessive number of degrees of freedom. For example, to integrate on the 2-dimesional sphere, approach 3(ii) would realize $S^2$ as a coadjoint orbit in the 3-dimensional $\mathfrak{so}(3)^*$, lift to $T^*SO(3)$ (6-dimensional) and embed this in $T^*\R^{3\times 3}$ (18-dimensional). What is wanted is an approach to constructing symplectic integrators  that leads to simple methods, that works for any $H$, and that uses few extra variables. We achieve this at the price of 
some extra work beforehand that depends on $P$. 


\section{Collective symplectic integrators for Lie--Poisson systems}

Throughout this section $M = (M,\omega)$ denotes a symplectic manifold and $P = (P,\{,\})$ denotes a Poisson manifold.

\subsection{Realizations and collective symplectic maps}

\begin{definition}
A \emph{realization} of $P$ is a Poisson map $\psi\colon M\to P$. 
A realization $\psi$ is called \emph{full} if $\psi$ is surjective.
If $(M,\omega)=(T^*\R^n,dq^i\wedge dp_i)$ then $(q^i,p_i)$ are called {\em canonical} or {\em Clebsch} variables for $P$. 
The \emph{fibres} of a realization $\psi$ are the subsets of $M$ given by $\psi^{-1}(\{x\})$ with $x\in P$.
\end{definition}

\begin{definition}
Let $\psi\colon M\to P$ be a realization of $P$.
A real valued function on $M$ of the form $H\circ\psi$ for some $H\colon P\to \R$ is called a \emph{collective Hamiltonian}. 
A map $\varphi\colon M\to M$ is {\em collective} if there is a map $\widetilde\varphi\colon P\to P$ such that $\widetilde\varphi\circ \psi = \psi \circ \varphi$. We say the map $\varphi$
{\em descends} to $\widetilde\varphi$ and that $\widetilde\varphi$ is the {\em reduced map} of $M$; similarly for vector fields.  
\end{definition}

Note that $\varphi\colon M\to M$ is collective if and only if it maps fibres of $\psi$ to fibres of $\psi$. If $\varphi$ is collective, the reduced map $\widetilde\varphi$ is only uniquely defined on the range of $\psi$.

Since a realization $\psi$ is a Poisson map, we have $ \{F\circ\psi,H\circ\psi\} = \{F,H\}\circ\psi$. The right hand side is collective, that is, constant on fibres. 
The left hand side is the Lie derivative along $X_{H\circ\psi}$ of the function $F\circ\psi$ that is constant on fibres. Therefore, the flow of $X_{H\circ\psi}$ maps fibres to fibres, or, put another way, is $\psi$-related to $X_H$ \cite{ma-ra}. If $\psi$ is surjective, the vector field $X_{H\circ\psi}$ descends to a vector field on $P$, namely $X_H$.

Realizations and collective Hamiltonians are a long-established tool in Hamiltonian dynamics. Key references are \cite{gu-st,li-ma,weinstein}  and especially \cite{ma-we} which is the main inspiration for our approach. Essentially all of the required geometry is in \cite{ma-we} and our contribution is to find conditions under which that framework can be useful for constructing integrators.

\begin{definition}
A {\em collective symplectic integrator} for $(P,\{,\},H)$ is a full realization of $P$ together with a symplectic integrator for $(M,\omega,H\circ \psi)$ that descends to a symplectic integrator for $(P,\{,\},H)$. A {\em collective symplectic map} for $(P,\{,\})$ is a full realization of $P$ together with a symplectic map of $(M,\omega)$ that descends to a symplectic (i.e. Poisson and fibre-preserving) map of $(P,\{,\})$.
\end{definition}  

In the most general case, this merely swaps one hard problem (constructing symplectic integrators on $P$) for three hard problems (finding full realizations, finding fibre-preserving symplectic integrators so that the integrator descends to $P$, and ensuring that the reduced integrator preserves the symplectic leaves). Note if $P$ has a Casimir $C$, then $C\circ\psi$ is a first integral of $X_{H\circ\psi}$, but preserving arbitrary integrals of a Hamiltonian system is difficult.

We now let $P=\mathfrak{g}^*$ with its Lie--Poisson structure. An action of a Lie group $G$  on $M$ is said to be (globally) Hamiltonian if it has a momentum map $J\colon M\to\mathfrak{g}^*$ that is equivariant with respect to the coadjoint action of $G$ on $\mathfrak{g}^*$, i.e.,
\begin{equation}
\label{eq:equi}
 J(g\cdot x) = \Ad_{g^{-1}}^* J(x)
\end{equation}
for all $x\in M$, $g\in G$ \cite{ma-ra}. $J$ is then Poisson and the image $J(M)$ is a union of (open subsets of) coadjoint orbits which are the symplectic leaves of $\g^*$. In this case $M$ is called a {\em Hamiltonian $G$-space}. 
In other words, any Hamiltonian action of $G$ on $M$ provides a realization $\psi\colon= J$ of $\mathfrak{g}^*$, although it need not be full.
Conversely, any Poisson map from $M$ to $\mathfrak{g}^*$ must be the momentum map for some
Hamiltonian group action on $M$ \cite{ma-we}. Thus, in the case $P=\mathfrak{g}^*$, we do not lose any generality by restricting attention in our search for realizations to Hamiltonian $G$-spaces.

The Hamiltonian vector fields of collective Hamiltonians can be calculated in this case using 
the result of \cite{gu-st} (page 215) that 
$X_{H\circ J}(x)=\xi_{dH(x)}$ where $\xi_p$ is the infinitesimal generator of the $G$-action associated with $p\in\g$; in vector notation, 
\begin{equation*}
 	X_{H\circ J}(x) = \Omega^{-1} (TJ(x))^T (\nabla H)(J(x)).
\end{equation*}

According to Weinstein \cite{weinstein}, the minimum dimension of a realization of a neighbourhood of $z\in\g^*$ is equal $\dim \g+\dim\g_z$, which can be as large as $2\dim\g$ (by taking $z=0$). He constructs such a representation, but it is not canonical. 

We will present in this section two general approaches to the construction of collective symplectic integrators. The first uses only the basic data $M$ and $J$ and seeks  conditions under which a symplectic integrator applied to $X_{H\circ J}$ is collective. We call this the direct approach. The second uses additional structure, a ``symmetry'' group $G_2$ that acts on the fibres of $J$; we call this the symmetry approach.

\subsection{The direct approach}

\begin{theorem}
\label{thm:onegroup}
Let $G$ be a Lie group with a Hamiltonian action on $M$ whose momentum map $J\colon M\to \mathfrak{g}^*$ has connected fibres. Let $\varphi\colon M\to M$ be a symplectic map.
\begin{enumerate}
\item[\rm (i)] If $\varphi$ maps $G$-orbits to $G$-orbits then it descends to a map on $ J(M)\subset\mathfrak{g}^*$ and that reduced map is (Lie--)Poisson. 
\item[\rm (ii)] If $\varphi$ fixes each orbit of $G$  then the reduced map preserves the coadjoint orbits.
\end{enumerate}
\end{theorem}

\begin{proof}\ \\[-5mm]
\begin{enumerate}
\item[(i)]
The map $\varphi$ preserves the orbits of $G$, so it preserves the distribution tangent to the orbits, $D_x := T(G\cdot x)$. The map $\varphi$  is symplectic, hence invertible, so $T\varphi|_{D_x}\colon D_x \to D_{\varphi(x)}$ is a linear isomorphism. The symplectic orthogonal to $D_x$ is the distribution tangent to the fibres of $J$ and is also preserved by $\varphi$, because $ \omega(T\varphi.u,T\varphi.v) = \omega(u,v)=0$ for all $v\in D_x$, $u\in D_x^\perp$, and $T\varphi.D_x = D_{\varphi(x)}$, so $T\varphi.u\in D_{\varphi(x)}^\perp$. Because the fibres of $J$ are assumed to be connected, $\varphi$ maps fibres of $J$ to fibres of $J$, that is, it descends to $J(M)$.
The reduced map $\tphi$ satisfies $J\circ\varphi = \tphi\circ J$. For all functions $F$, $G$ on $\mathfrak{g}^*$ we have
$
\{F,G\}\circ\tphi\circ J = \{F,G\}\circ J\circ \varphi
= \{F\circ J,G\circ J\}\circ\varphi
= \{F\circ J\circ \varphi,G\circ J\circ\varphi\}
= \{F\circ\tphi\circ J,G\circ\tphi\circ J\}
= \{F\circ \tphi,G\circ\tphi\}\circ J$
and therefore $\{F,G\}\circ\tphi=\{F\circ \tphi,G\circ\tphi\}$, that is, $\tphi$ is (Lie--)Poisson.
\item[(ii)] If $\varphi$ fixes each orbit of $G$ then $\varphi(x)$ lies in the orbit $G\cdot x$, that is, $\varphi(x) = g\cdot x$ for some $g\in G$. Therefore from (\ref{eq:equi}), $J(\varphi(x))=\Ad^*_{g^{-1}} J(x)$ and the reduced map preserves the coadjoint orbits.
\end{enumerate}
\end{proof}

\newcommand*\meth{\widetilde{\varphi}}
\newcommand*\Adg[1]{\Ad^*_{#1}}
\newcommand*\inv{^{-1}}

If an integrator $\varphi$ defines a map from vector fields $X$ to diffeomorphisms $\varphi(X)$, the integrator is said
to be {\em $G$-equivariant} if this map is $G$-equivariant with respect to the action of $G$ on vector fields and on diffeomorphisms, i.e., if $\varphi(Tg\inv .X\circ g) = g^{-1}\circ \varphi(X)\circ g$ for all $g\in G$. For example, all Runge--Kutta methods are affine-equivariant. If the action of $G$ is symplectic, then $Tg\inv .X_H\circ g = X_{H\circ g}$ so in this
case $\varphi(X_{H\circ g} = g^{-1}\circ \varphi(X_H)\circ g$. Our next result shows that equivariance descends to $J(M)$.

\begin{theorem}
	Let $G$ be a Lie group with Hamiltonian action on $M$ with momentum map $J$.
	For a vector field $X$, we denote by $\varphi(X)$ the corresponding diffeomorphism for the method $\varphi$.
	Assume that $\varphi$ is a $G$-equivariant integrator on $M$ such that, for each Hamiltonian $H$ defined on $J(M)$, $\varphi(X_{H\circ J})$ descends on $J(M)$ to a map denoted $\meth(H)$.
	Then the map $\meth$ is \emph{equivariant} in the sense that 
	$\meth(H \circ \Ad^*_{g^{-1}}) = \Ad^*_{g} \circ \meth(H) \circ \Ad^*_{g^{-1}}$
	for any $g \in G$.
\end{theorem}
\begin{proof}
$G$ acts as a Poisson map, so $X_{H \circ J \circ g} = T g^{-1} X_{H \circ J}\circ g$ \cite[Prop.~10.3.2]{ma-ra}.
Combining the equivariance of $\varphi$ and the equivariance of $J$ we obtain 
\begin{equation}
	\label{eq:phiJequi}
\varphi(X_{H \circ \Adg{g\inv} \circ J}) = g\inv \circ \varphi(X_{H\circ J}) \circ g
.
\end{equation}

Now, $\meth$ is defined by $\meth(H) \circ J = J \circ \varphi(X_{H\circ J})$.
This gives $\meth(H\circ\Adg{g\inv}) \circ J = J \circ \varphi(X_{H\circ\Adg{g\inv}\circ J})$.
Using (\ref{eq:phiJequi}) and the equivariance of $J$ we finally obtain $\meth(H\circ\Adg{g\inv}) \circ J = \Adg{g} \circ \meth(H) \circ \Adg{g\inv} \circ J$.
We conclude that the result must hold on $J(M)$.
\end{proof}

A consequence is that the reduced integrator will preserve any coadjoint symmetries of $H$.

\begin{definition} Let $G$ act on $M$. A function $I\colon M\to \R^k$ is an {\em invariant} of $G$ if $I(g\cdot x) = I(x)$ for all $x\in M$. The function $I$ is a {\em complete invariant} of the action of the orbits of $G$ are given by the fibres of $I$.
\end{definition}

\begin{theorem}
\label{thm:onegroupint}
Let $G$ be a Lie group with a Hamiltonian action on $M$ with complete invariant $I\colon M\to\R^k$ and whose momentum map $J\colon M\to \mathfrak{g}^*$ has connected fibres.
Let $\varphi\colon M\to M$ be a symplectic map.
\begin{enumerate}
\item[\rm (i)] If $I$ is a first integral of $\varphi$ then the reduced map of $\varphi$ on $\mathfrak{g}^*$ is (Lie--)Poisson and preserves the coadjoint orbits.
\item[\rm(ii)] If $M$ is a symplectic vector space and $I$ is quadratic then symplectic Runge--Kutta methods applied to $X_{H\circ J}$ are collective symplectic integrators for $(J(M), \{,\},H)$. 
\item[\rm(iii)] If $M$ is a symplectic vector space and there is a Darboux basis in which $I$ is bilinear, then symplectic partitioned Runge--Kutta methods in this basis are collective symplectic integrators for $(J(M), \{,\},H)$. 
\end{enumerate}
\end{theorem}

\begin{proof}\ \\[-5mm]
\begin{enumerate}
\item[(i)] By hypothesis, $\varphi$ preserves $I$, hence it fixes each group orbit as required by Theorem \ref{thm:onegroup}.
\item[(ii)] Symplectic Runge--Kutta methods preserve all quadratic first integrals of Hamiltonian systems \cite{ha-lu-wa}, hence under these hypotheses they fix each group orbit.
\item[(iii)] Partitioned symplectic Runge--Kutta methods preserve all bilinear first integrals of Hamiltonian systems \cite{ha-lu-wa}, hence under these hypotheses they fix each group orbit.
\end{enumerate}
\end{proof}

If the realization is full then these integrators provide collective symplectic integrators for $(\mathfrak{g}^*,\{,\},H)$.

The following example shows that Runge--Kutta methods cannot preserve group orbits in general. Let $a\in{\mathbb C}$ and let $G=\R^{>0}$ act on $M=\mathbb{C}$ by $t\cdot z = {\rm e}^{a t}z$; consider any vector field tangent to the orbits. When $a$ is imaginary, the orbits are the origin and the circles $|z|^2=$ const. and are fixed by symplectic Runge--Kutta methods. When $a$ is real, the orbits are the origin and the straight open rays meeting the origin. This action has no smooth invariants (the orbit closures intersect at 0, so any continuous invariant must be constant). Each line through the origin is invariant, and is invariant under all Runge--Kutta methods; each contains an invariant open ray, which is preserved by `positivity-preserving' Runge--Kutta methods. When $a$ is neither real nor imaginary, the orbits are spirals; there are no smooth invariants and the orbits are not fixed by any Runge--Kutta method. Thus, the proposed integration method can only cope with fairly simple actions. This example motivates the following extension of Theorem \ref{thm:onegroupint}.

A {\em polyhedral set} is the intersection of affine subspaces and closed half-subspaces. An example is the orthant $x_i\ge 0$ for all $i$. Sufficient conditions for a Runge--Kutta method to preserve this orthant for sufficiently small time steps when it is invariant are known \cite{horvath}; such methods are called {\em positivity-preserving}. The midpoint rule is positivity-preserving; positivity is preserved for time steps less than $2/L$ where $L$ is the Lipschitz constant of the vector field.

\begin{theorem}
\label{thm:polytope}
Let $G$ be a Lie group with a Hamiltonian action on the symplectic vector space $M$ whose momentum map $J\colon M\to \mathfrak{g}^*$ has connected fibres and which has quadratic (resp. bilinear) invariants  $I\colon M\to\R^k$ such that the closure of each orbit is the intersection of an invariant polyhedral set and a fibre of $I$. Then for sufficiently small time steps, positivity-preserving symplectic Runge--Kutta methods applied to $X_{H\circ J}$ are collective symplectic integrators for $(J(M), \{,\},H)$. 
\end{theorem}

\begin{proof}
Positivity-preserving Runge--Kutta methods preserve also preserve invariant polyhedral sets for sufficiently small time steps \cite{horvath}. 
 Quadratic invariants are preserved by symplectic Runge--Kutta methods. Therefore, the closure of each orbit is fixed by the given methods. The boundary of an orbit closure is itself, by hypothesis, the intersection of an invariant polyhedral set and a fibre of $I$ and hence is itself invariant. Since the map is invertible, the interior of the orbit closure is also invariant. This is either an orbit (in which case we are done) or a union of orbits (in which case the argument is repeated). Thus every orbit is fixed by the method and Theorem \ref{thm:onegroup} gives the result.
\end{proof}

If the orbit closures are just slightly more general, for example the intersection of a polyhedral set and a fibre of an invariant quadratic, then a positivity-preserving symplectic Runge--Kutta method need not fix them. (In $\R^3$, the circle $x_1=0$, $\|x\|^2=1$ in the flow of $\dot x = J(x)x$, $J^T=-J$, $J_{1k}=f_k(x_1,\|x\|^2-1)$, $f(0,0)=0$ is such a set.)

\subsection{The symmetry approach}

Often there is a second group acting on the fibres of $J$. If its orbits are large enough, and respected by the integrator, this can be enough to ensure that the integrator is collective.

\begin{theorem}
\label{thm:transitive}
Let $G_1$ be a group with a Hamiltonian action on $M$ and momentum map $J_1$.
Let $G_2$ be a group acting on $M$ that fixes each fibre of $J_1$ and is transitive on them. Then any $G_2$-equivariant symplectic map is collective.
\end{theorem}
\begin{proof}
By transitivity, any two points on the fibre $J_1^{-1}(z)$ may be written in the form $x$, $g_2\cdot x$ for some $g_2\in G_2$. Then $G_2$-equivariance gives $\varphi(g_2\cdot x) = g_2\cdot \varphi(x)$, that is, $\varphi(g_2\cdot x)$ lies on the $G_2$-orbit (= $J_1$-fibre) of $\varphi(x)$. Thus the $J_1$-fibres map to $J_1$-fibres and the map descends; from equivariance of $J_1$ it is collective.
\end{proof}

Consider $G=GL(1)$ acting by cotangent lifts on $T^*\R$ with momentum map $J=qp$. $G$ acts on the fibres of $J$, and is transitive on generic fibres, but is not transitive on $J^{-1}(0)$. However, we shall see its orbits are still large enough to construct collective integrators. This motivates the following definition.

\begin{definition}
A smooth function $\Psi\colon M\to P$ has {\em complete symmetry group $G$} if $G$ acts smoothly on $M$, has invariant $\Psi$,  and there are no non-constant $G$-invariant continuous functions on any fibre of $\Psi$.
\end{definition}

If $\Psi$ is a complete invariant for $G$, i.e., if the fibres of $\Psi$ are the orbits of $G$, then $G$ is a complete symmetry group for $\Psi$.
If the action of $G$ is transitive on a dense subset of each fibre of $\Psi$, then $G$ is a complete symmetry group for $\Psi$.
If the orbits of $G$ are closed (in particular if the action of $G$ is proper, which happens if $G$ is compact), then the property of being a complete symmetry group is equivalent to transitivity of the action of $G$ on the fibres.

\begin{lemma}
\label{lemma:big}
Let $\Psi\colon M\to P$ have complete symmetry group $G$.  Then any continuous map $\varphi\colon M\to M$ which maps orbits of $G$ to orbits of $G$ is $\Psi$-collective.
\end{lemma}
\begin{proof}
		Let $x\in \Psi(M)$ and consider one fibre $L := \Psi^{-1}(x)$.
		Let $N := \Psi^{-1}( \Psi(\varphi(L)))$ be the set of fibres of $\Psi$ in $\varphi(L)$, and quotient this space by 
			the fibres to obtain $N/ \Psi$.
		Consider a continuous function $g\colon N/ \Psi\to\R$.
		Pulling back $g$ by $ \Psi$ gives a continuous function $h$ on $N$ which is constant on the orbits of $G$.
		Therefore $h|_L$ is also constant on the orbits of $G$.
		Because $G$ is a complete symmetry group of $ \Psi$, $h|_L$ must be constant.
		The function $g$ must then be constant, because of the definition of $N$.
		Since $N/ \Psi$ is Hausdorff (because the fibres are closed subsets of $M$), and all its continuous functions 
			are constant, it must reduce to one point, that is, $N$ consists of exactly one fibre.
		We conclude that $\varphi$ maps fibres of $ \Psi$ to fibres of $\Psi$, that is, $\varphi$ is $ \Psi$-collective.
\end{proof}

For example, if $G_1$ has a Hamiltonian action on $M$ and momentum map $J_1$ with complete symmetry group $G_2$, Then any $G_2$-equivariant symplectic map is collective. Our main example of this situation is the following.

\begin{theorem}
\label{thm:twogroupsint}
Let $G_1$ be a group with a Hamiltonian action on $M$ and momentum map $J_1$ with complete symmetry group $G_2$ that has momentum map $J_2$. Let $H$ be a Hamiltonian on $\mathfrak{g}_1^*$.
\begin{enumerate}
\item[\rm(i)]
If the action of $G_2$ is linear  then  symplectic Runge--Kutta methods applied to $X_{H\circ J_1}$ are
collective symplectic integrators for $(J_1(M), \{,\},H)$.
\item[\rm(ii)]
If the action of $G_2$ is a linear cotangent lift then symplectic partitioned Runge--Kutta methods applied to $X_{H\circ J_1}$ are collective symplectic integrators for $(J_1(M), \{,\},H)$.
\end{enumerate}
\end{theorem}

\begin{proof}
First note that $H\circ J_1$ is constant on fibres of $J_1$, hence $G_2$-invariant.
\\[-5mm]
\begin{enumerate}
\item
The linear symmetry $G_2$ of $X_{H\circ J_1}$ is preserved by symplectic Runge--Kutta methods.  Lemma \ref{lemma:big} now gives the result.
\item
The linear cotangent lift symmetry $G_2$ of $X_{H\circ J_1}$ is preserved by symplectic partitioned Runge--Kutta methods.  Lemma \ref{lemma:big} now gives the result.
\end{enumerate}
\end{proof}

Often there is a complete symmetry group $G_2$ that makes $(G_1,G_2)$ into a {\em dual pair}.

\begin{definition} \label{defn:dualpair}
A $C^k$ (resp. {\em analytic}) {\em dual pair} is a pair $G_1$, $G_2$ of Lie groups  with Hamiltonian actions on $M$ and momentum maps $J_1$, $J_2$ such that
the $C^k$ (resp. analytic) functions that commute with $J_1$-collective functions are collective for $J_2$ and vice versa. That is, the sets of
$J_1$-collective and $J_2$-collective functions are mutual centralizers in the Poisson algebra of all $C^k$ (resp. analytic) functions on $M$.
\end{definition}

For a dual pair, the orbits of $G_1$ are contained in the fibres of $J_2$ and the orbits of $G_2$ are contained in the fibres of $J_1$.
The Hamiltonian vector field $X_{F\circ J_1}$ is $G_2$-equivariant and the Hamiltonian vector field $X_{F\circ J_2}$ is $G_1$-invariant. 

Note that if $f$ is constant on $G_2$-orbits, then $\langle {\rm d}f,X_{J_2}\rangle = 0 = \{f,J_2\}$, so the dual pair condition ensures that $f$ is $J_1$-collective. That is, $G_2$ is a complete symmetry group for $J_1$, so $G_2$-equivariant symplectic integrators are collective.

\begin{theorem}
\label{thm:dualpair}
Let $G_1$, $G_2$ be a $C^k$ or analytic dual pair acting on $M=\R^n$ and let $H\colon M\to \R$ be a smooth Hamiltonian. If the action of $G_2$ is linear (cotangent lift) then (partitioned) symplectic Runge--Kutta methods produce collective symplectic integrators that descend to $C^{k-1}$ or analytic symplectic integrators on $(J_1(M), \{,\},H)$.
\end{theorem}
\begin{proof}
The integrator preserves $J_2$. Differentiation with respect to the time-step $h$ represents the integrator as
the time-$h$ flow of a nonautonomous Hamiltonian vector field on $M$; let $\widetilde H(x,t)$ be the Hamiltonian.
Because its flow preserves $J_2$, $\{\widetilde H, J_2\}=0$. The dual pair condition implies that $\widetilde H = \widehat H \circ J_1$, where $\widehat H$ is $C^k$ or analytic; thus its flow, the reduced integrator, is $C^{k-1}$ or analytic.
\end{proof}

Definition \ref{defn:dualpair} is an instance of Weinstein's definition \cite{weinstein} of dual pair, which involves two Poisson maps $\psi_i$ from $M$ to Poisson manifolds $P_i$. There are many other versions and refinements of the concept of dual pair; see \cite{ka-le,le-mo-sj,or-ra}. One that is of use in constructing integrators for classical Lie--Poisson manifolds is the original representation-theoretic dual pair of Howe \cite{howe},  two subgroups $G_1$, $G_2$ of $Sp(M)$ such that $G_1$ (resp. $G_2$) is the centralizer of $G_2$ (resp. $G_1$). When the $G_i$ are reductive (i.e., when every $G_i$-invariant subspace of $M$ has a $G_i$-invariant complement, $i=1,2$),  the quadratic functions $J_2$ form a generating set for the $G_1$-invariant polynomials. If $H$ is any polynomial that satisfies $\{H,J_1\}=0$, then $H = \widehat H(J_2)$ where $\widehat H$ is a polynomial. That is, one has a `First Fundamental Theorem'---a set of polynomials that generate all polynomial invariants---for reductive Howe dual pairs \cite{howe,ka-le}. Consequently, these $(G_1,G_2)$ form analytic dual pairs:

\begin{theorem}
\label{thm:analytic}
Let $(G_1,G_2)$ be subgroups of $Sp(M)$ such that $J_1$ (resp. $J_2$) forms a generating set for the $G_2$ (resp. $G_1$)-invariant polynomials on $M$. Then $(G_1,G_2)$ forms an analytic dual pair.
\end{theorem}
\begin{proof}
Let $H$ be real analytic at 0 and let $\{H,J_1\}=0$ in a neighbourhood of 0. Expand $H=\sum_k H_k$ in a Taylor series where $H_k$ is homogeneous of degree $k$. Then $\{H_k,J_1\}$ is homogeneous of degree $k$ and thus $\{H_k,J_1\}=0$. By assumption, there are polynomials $\widehat H_k\colon\mathfrak{g}_2^*\to\R$ such that $H_k = \widehat H_k\circ J_2$. Because $H$ is analytic at 0, $\sum H_k(x) = \sum (\widehat H_k\circ J_2)(x)$ is convergent for $x$ in some neighbourhood of 0 in $M$, hence $\widehat H := \sum \widehat H_k$ is convergent at $J_2(x)$ and hence convergent in some neighbourhood of the origin in $\mathfrak{g}_2^*$. That is, $H = \widehat H\circ J_2$ where $\widehat H$ is analytic.
\end{proof}

Note that such $(G_1,G_2)$ need not form $C^\infty$ dual pairs. Consider $(GL(1),GL(1))$ with $J_1=J_2=qp$. Then ${\rm e}^{-(qp)^{-2}}{\rm sign}(q)$ commutes with $J_1$ but is not a function of $J_2$.

\section{Collective integrators from bifoliations}

We have seen that we can construct symplectic integrators on Poisson manifolds from standard symplectic integrators in some cases from a Hamiltonian group action (Theorems \ref{thm:onegroup}, \ref{thm:onegroupint}, \ref{thm:polytope}) and from two commuting Hamiltonian group actions (Theorem \ref{thm:dualpair}). Theorem \ref{thm:dualpair} is a special case of Theorem \ref{thm:onegroup}, but it gives a little more, namely, it specifies the invariants that need to be preserved, it provides Poisson integrators for two spaces (corresponding to the two groups), and it relates the construction to the standard examples of Howe dual pairs. One can ask if it is possible, instead, to {\em generalize} Theorem \ref{thm:onegroup}, being prepared, of course, to get a little less in return. The key structures are the foliations defined by the orbits and by the fibres of the momentum map (in Theorem \ref{thm:onegroup}) or the fibres of each momentum map (in Theorem \ref{thm:dualpair}). These generalize to a {\em bifoliation}, a structure discussed in \cite{fasso,li-ma,vaisman} whose properties we summarize briefly here.

 Let $\F$ be a foliation of a symplectic manifold $(M,\omega)$.  The polar of $\F$, if it exists, is the unique 
foliation $\F^\perp$ of $M$ such that the tangent spaces of its leaves 
are the symplectic orthogonals of the tangent spaces of the leaves of 
$\F$. A foliation which has a polar is called symplectically complete, and $(\F,\F^\perp)$ is called a bifolation of $M$.
A foliation is symplectically complete iff the Poisson bracket of any two functions that are constant 
on leaves is again constant on leaves. In the case that $M/\F$ is a manifold, $\F$ is symplectically
complete iff there is a Poisson structure on $M/\F$ such that the projection map $\psi\colon M\to M/\F$ is Poisson; such a structure
is unique \cite{fasso}. 

\begin{theorem}
\label{thm:nogroup}
Let $(\F,\F^\perp)$ be a bifoliation of the symplectic manifold $(M,\omega)$ such that $M/\F$ is a manifold. Let $\psi\colon M\to M/\F$ be the projection and let $H\colon M/\F\to\R$. Then $X_{H\circ\psi}$ fixes each leaf of $\F^\perp$ and any symplectic integrator for $H\circ \psi$ that fixes each leaf of $\F^\perp$ descends to a symplectic integrator of $X_H$ in the Poisson manifold $M/\F$. In particular, when $M$ is a symplectic vector space and $\F^\perp$ is the set of fibres of a set of quadratic functions, then symplectic Runge--Kutta methods applied to $X_{H\circ\psi}$ descend to symplectic integrators of $X_H$.
\end{theorem}

\begin{proof}
Any symplectic map that preserves $\F$ (resp. $\F^\perp$) necessarily preserves $\F^\perp$ (resp. $\F$) and hence descends to a Poisson map on $M/\F$ and on $M/\F^\perp$. The crucial step is to ensure that the reduced map on $M/\F$ fixes each coadjoint orbit. This follows from Proposition 14.21 of \cite{li-ma} which states that the image under $\psi$ of a leaf of $\F^\perp$ is contained in a symplectic leaf of $M/\F$. Thus, any symplectic map that fixes each leaf of $\F^\perp$ must fix each coadjoint orbit in $M/\F$.
\end{proof}

Theorem \ref{thm:nogroup} is the most general case; however, it gives somewhat less than the previous Theorems 
because the Poisson manifold is constructed in a somewhat convoluted way from the quadratic functions and the
projection $\psi$ is defined abstractly. In practise one will also need the leaves of $\F$ to be the fibres of functions
that can be used to define $\psi$.


There is a special case of Theorem \ref{thm:nogroup} which is particularly nice and which produces minimum-dimensional realizations and minimum-dimensional collective symplectic integrators. It uses the following construction of
Nekhoroshev \cite{nekhoroshev}; see also \cite{fasso}. Let $f_1,\dots,f_k$ be $k$ commuting functions on $(M,\omega)$. Then their fibres are coisotropic and symplectically complete. Their polar is isotropic and their projection to the quotient of $M$ by the polar lie in its symplectic leaves.
In particular, if $(f_1,\dots,f_{2n-k}):M\to\R^{2n-k}$ is a submersion and
$$ \{f_i,f_j\} = 0,\quad i=1,\dots,k,\quad j=1,\dots,2n-k$$
then the $f_i$ define a bifoliation with $\F$ the fibres of $f_1,\dots,f_{2n-k}$ isotropic; $\F^\perp$ the fibres of $f_1,\dots,f_k$ coisotropic; and $f_1,\dots,f_k$ the lifts of the Casimirs of the Poisson manifold $M/\F$, whose symplectic leaves have dimension
$2n-2k$. Moreover, we may take $f_1,\dots,f_{2n-k}$ as local coordinates on $M/\F$.
In our application, $f_1,\dots,f_k$ should be quadratic. Moreover, we can even drop the `collective' and consider any Hamiltonian with first integrals $f_i$:

\begin{theorem}
\label{thm:isotropic}
Let $M$ be a symplectic vector space and let $f_1,\dots,f_k$ be $k$ commuting quadratic functions on $M$. Then their Hamiltonian vector fields $X_{f_i}$ are integrable and generate an abelian group action (with momentum map $f_1,\dots,f_k$) whose orbits form the isotropic foliation $\F$ polar to the fibres of $f_1,\dots,f_k$. Symplectic Runge--Kutta methods applied to Hamiltonians with first integrals $f_1,\dots,f_k$ (i.e., such that $\{H,f_i\}=0$ for $i=1,\dots,k$) descend to symplectic integrators on the Poisson manifold $M/\F$. If $\F$ is the fibres of the functions $f=(f_1,\dots,f_{2n-k})$ then symplectic Runge--Kutta methods applied
to $H\circ f$ form collective symplectic integrators.
\end{theorem}

\section{Examples}

The examples are arranged in order of increasing dimension, starting with the 2-dimensional nonabelian Lie algebra. Recall that all cotangent lifted actions, and all linear symplectic actions, are Hamiltonian \cite{ma-ra}. We use momentum maps that are Poisson for the ``+'' Lie--Poisson bracket; the ``$-$'' bracket can be obtained by changing the sign of the momentum map.

\begin{example}\rm{\bf (a)}
\label{ex1}
Let $G=A(1)$, the group of  affine transformations of $\R$, the smallest nonabelian Lie group. Let $G$ act on $M=T^*\R$ by cotangent lifts, i.e., $(a,b)\cdot(q,p) = (aq+b,p/a)$, $a$, $b\in\R$. Then $J(q,p)=(qp,p)$. Let $(w_1,w_2)$ be coordinates on $\mathfrak{g}^*$. There is a 2-dimensional coadjoint orbit $\{(w_1,w_2)\colon w_2\ne 0\}$ and many 0-dimensional coadjoint orbits, $\{(w_1,0)\}$ for each $w_1$.  The fibres of $J$ consist of the $q$-axis and the points off the $q$-axis; they are connected. $J(M)$ is not quite all of $\g^*$: it consists of the origin together with the two half-planes. The vector field $X_{H\circ J}$ is
$$\eqalign{
 \dot q &= q \frac{\partial H}{\partial w_1}(qp,p) + \frac{\partial H}{\partial w_2}(qp,p)\cr
 \dot p &= -p\frac{\partial H}{\partial w_1}(qp,p)}$$
which is the generator of the group action corresponding to $\nabla H$.
Any partitioned symplectic Runge--Kutta method that fixes the group orbits (the $q$-axis and its complement) will, from
Theorem \ref{thm:onegroup}(ii), provide a collective symplectic integrator for $J(M)$. The orbit closures are the polyhedral sets $\{(q,p)\colon p\ge 0\}$, $\{(q,p)\colon p\le 0\}$, and $\{(q,0)\}$, and thus from Theorem \ref{thm:polytope}, the midpoint rule fixes the group orbits for sufficiently small time steps.
\\ \rm{\bf (b)}
An  action of $G$ for which $J(M)=\g^*$ can be constructed by prolonging the above action to act diagonally on $T^*\R^2$. The momentum map $J(q,p) = (p\cdot q, p_1+p_2)$ is surjective and has connected fibres. There are two bilinear invariants, $I_1:=(q_2-q_1)p_1$ and $I_2:=(q_2-q_1)p_2$, which classify generic orbits, and which are preserved by partitioned symplectic Runge--Kutta methods. However, the fibre with $I_1=I_2=0$ is 3-dimensional and contains several orbits, 
$$\eqalign{
& \{(q_1,q_1,p_1, p_2)\colon q_1\in\R,\  p_2/p_1= {\rm const.}\},\cr
& \{(q_1, q_2, 0,0)\colon q_1\ne q_2\},\cr
& \{(q_1, q_1,0,0)\colon q_1\in\R\}.
}$$
Their closures are polyhedral sets, so from Theorem \ref{thm:polytope} the midpoint rue  is a collective symplectic integrators for $(\mathfrak{g}^*,\{,\},H)$.
\end{example}

\begin{example}{\bf (a)}\rm
\label{ex2}
Let $G_1=O(3)$ with its natural cotangent lifted action on $M=T^*\R^3$. Its momentum map $J_1 = q\times p$ is surjective. The fibre of $J_1$ through $(q,p)$ is $\{(a q + b p, c q + d p)\colon a d - b c = 1\}$, which is connected.
The invariants of $G_1$ are generated by $(q\cdot q,p\cdot p, q\cdot p)$ which are quadratic and form a complete set of invariants. Therefore symplectic Runge--Kutta methods such as the midpoint rule applied to 
$X_{H\circ J_1}$, namely
$$\eqalign{
\dot q &= -q\times \nabla H(q\times p),\cr
\dot p &= -p\times \nabla H(q\times p)}$$
generate Lie--Poisson integrators on $\mathfrak{o}(3)^*$ for any $H$. The lifted coadjoint orbits are $\|q\times p\|^2=$ const., which are quartic invariants of $X_{H\circ J_1}$. However, from Theorem 1, we know that they are conserved by the integrator. 

The Hamiltonian vector fields of the invariants suggest that $J_1$ has a complete symmetry group
$G_2 = SL(2)$ with $J_2 = (q\cdot q,p\cdot p, q\cdot p)$, and linear action
$$\left(\matrix{a & b \cr c & d}\right)
\cdot(q,p) = (a q + b p,c q + d p)$$
which is the natural action of $SL(2)$ on $(T^*\R)^3$. The lifted Casimir
$$\|q\times p\|^2 = \|J_1(q,p)\|^2 = (q\cdot q)(p\cdot p)-(q\cdot p)^2$$
is collective for $J_2$---another way of seeing why it is conserved by the integrator.
\\ {\bf (b)}
It is interesting to ``dualize'' this example by considering Hamiltonians $H\circ J_2$. 
These are $O(3)$-invariant Hamiltonians such as classical central-force Hamiltonians $\|p\|^2/2+V(\|q\|^2)$. The invariants of $G_2$ are $J_1(q,p)=q\times p$ which are bilinear. The action of $G_1$ is transitive on the fibres of $J_2$ so we can use the symmetry group approach. The action of $G_1$ is a linear cotangent lift, hence from Theorem \ref{thm:transitive} partitioned symplectic Runge--Kutta methods, such as leapfrog, yield collective integrators for $J_2(M)\subset\mathfrak{sl}(2)^*$ on $J_2(M)$. Note $J_2$ is not surjective. The Casimir in $\mathfrak{sl}(2)^*\ni w$ is $C := w_1 w_2 - w_3^2$ and $C\circ J_2 = \|q\times p\|^2 \ge 0$, so $J_2(M) = \{w\in \mathfrak{sl}(2)^*: w_1\ge 0, \ w_2\ge 0, \ 
C \ge 0\}$ which is a solid cone. Figure \ref{fig:sl2} shows some orbits of Hamiltonian systems on $\mathfrak{sl}(2)^*$ (and
the coadjoint orbits they inhabit)
calculated using the collective leapfrog method. \end{example}

Example \ref{ex2}(b) is closely related to the symplectic reduction of $O(3)$-invariant Hamiltonians $H\circ J_2$. Indeed, {\em all} $O(3)$- (and all $SO(3)$-) invariant Hamiltonians are collective for $J_2$. Most treatments of reduction for this example (see, e.g., \cite{ma-ra}) do not involve $\mathfrak{sl}(2)^*$, the dual of the algebra of invariants, instead passing to a symplectic reduced space with coordinates $(\|p\|^2,q\cdot p)$ or ($\|q\|$,$\|p\|$). See \cite{le-mo-sj} for a treatment that features the algebra of invariants. A feature of the $\mathfrak{sl}(2)^*$ approach is that is allows one to see the relationship between orbits of different angular momentum and to visualize the orbits as intersections of energy and Casimir level sets.

\begin{figure}
\begin{center}
\includegraphics[width=8cm]{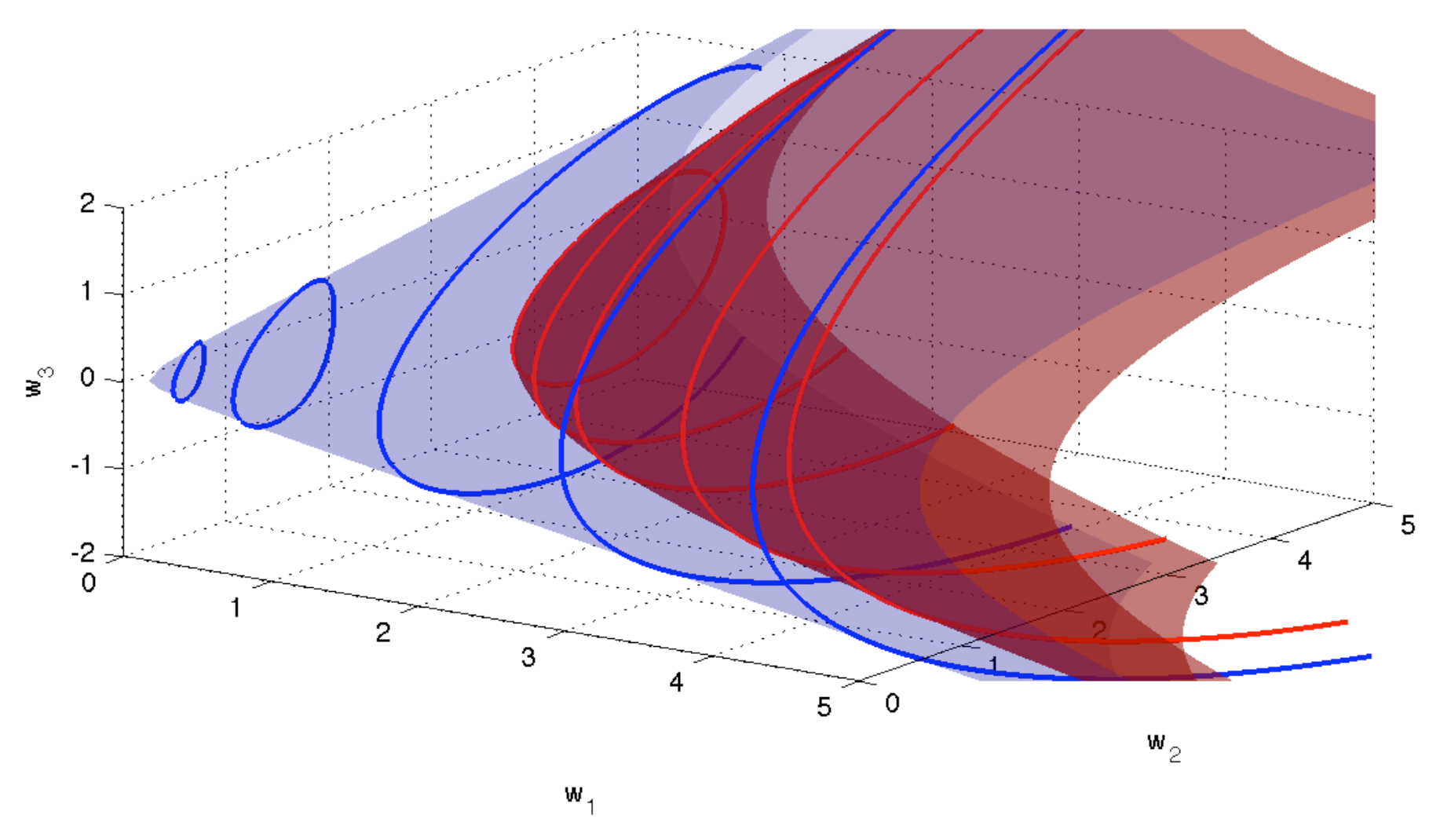}
\includegraphics[width=8cm]{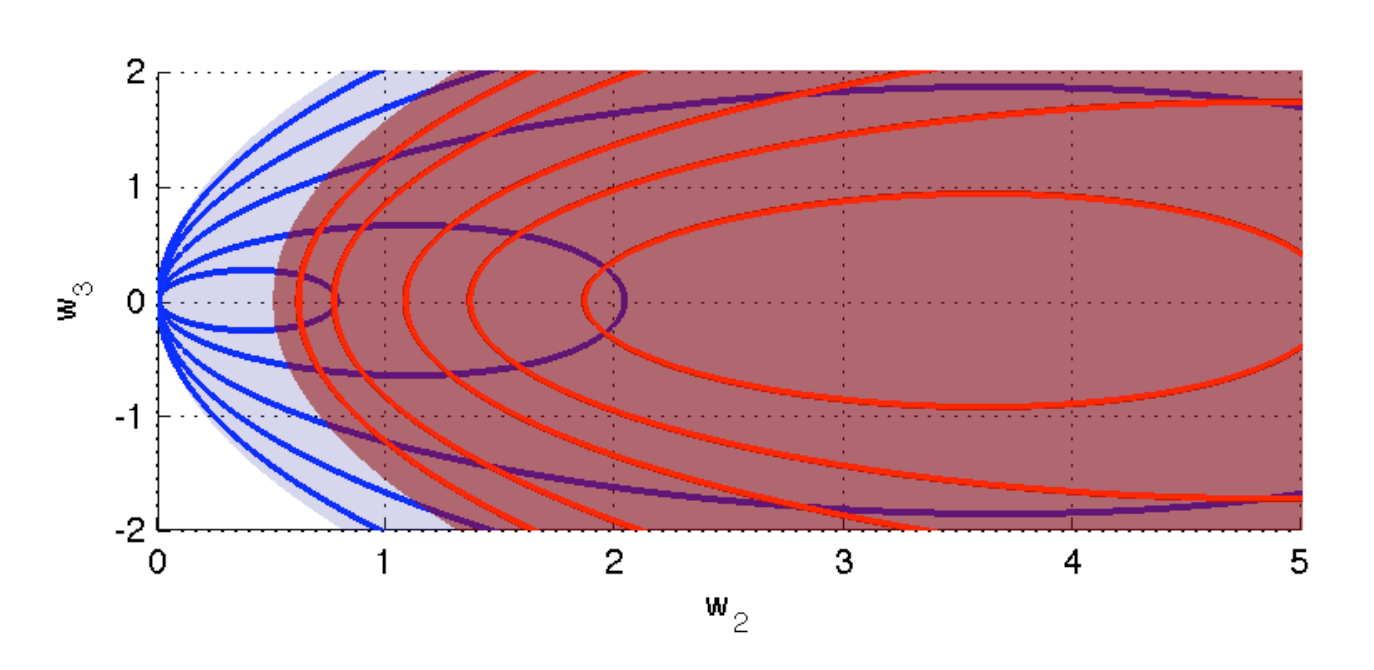}
\includegraphics[width=8cm]{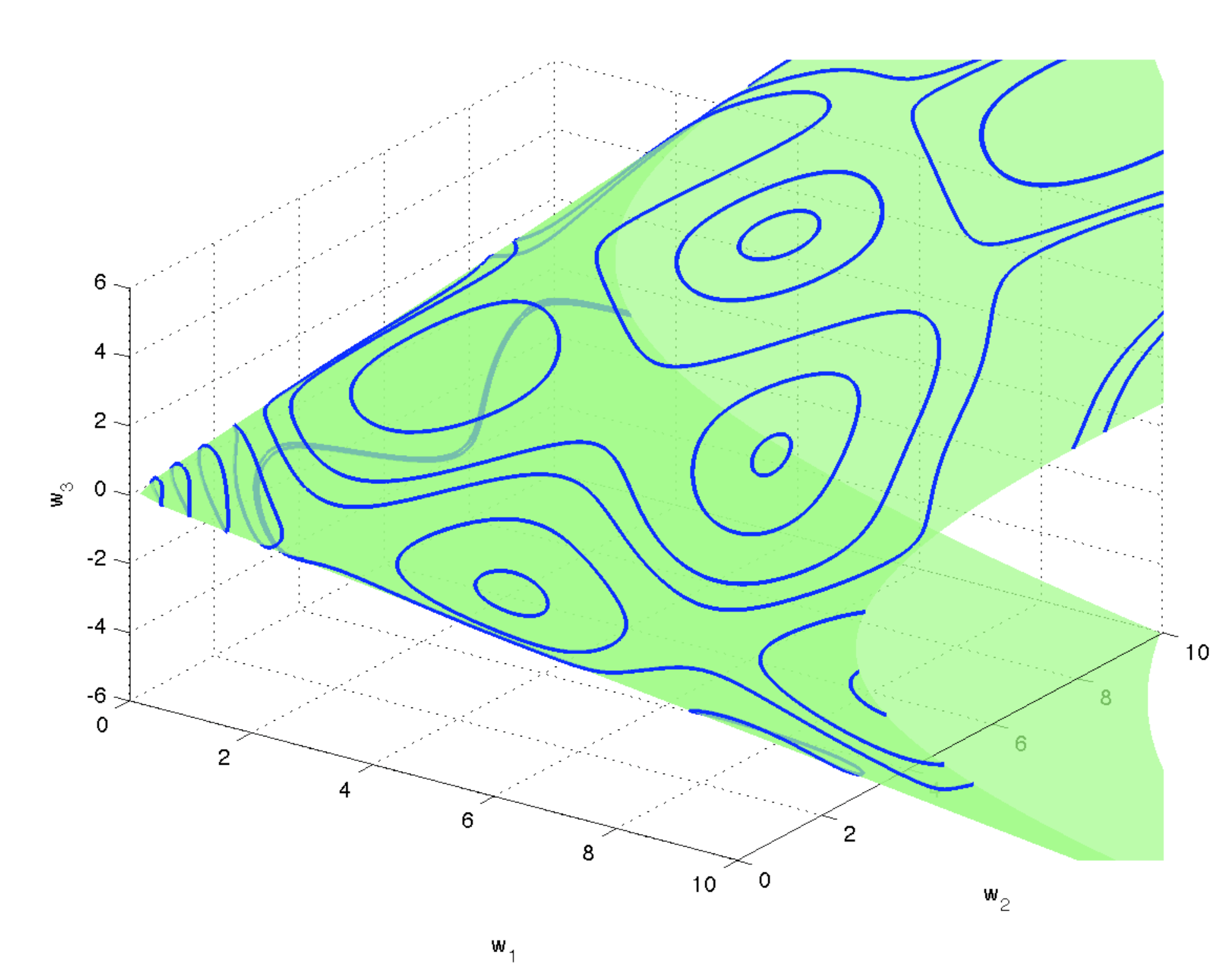}
\caption{\label{fig:sl2} Some orbits of Hamiltonian systems on $\mathfrak{sl}(2)^*$ and
the coadjoint orbits they inhabit. The coadjoint orbits are given by $w_1 w_2 - w_3^2 = C$; the cone $C=0$ and hyperboloid $C=2.5$ are shown. Top: Hamiltonian $H(w_1,w_2,w_3) = w_1 + \frac{1}{2}w_1^2 + \frac{1}{2}w_2$. The collective Hamiltonian for $J = (q\cdot q,p\cdot p,q\cdot p)$ is $H\circ J = \|q\|^2 + \frac{1}{2}\|q\|^4 + \frac{1}{2}\|p\|^2$. Calculated with the collective leapfrog method with $\Delta t = 0.01$. Middle: as before, showing the `standard' reduced variables $w_2$, $w_3$. Bottom: Hamiltonian 
$\sum_{i=1}^3 \cos w_i$ on $C=0$. The intersection of the energy and Casimir level sets creates a complex phase portrait.}
\end{center}
\end{figure}

\begin{figure}
\begin{center}
\includegraphics[width=8cm]{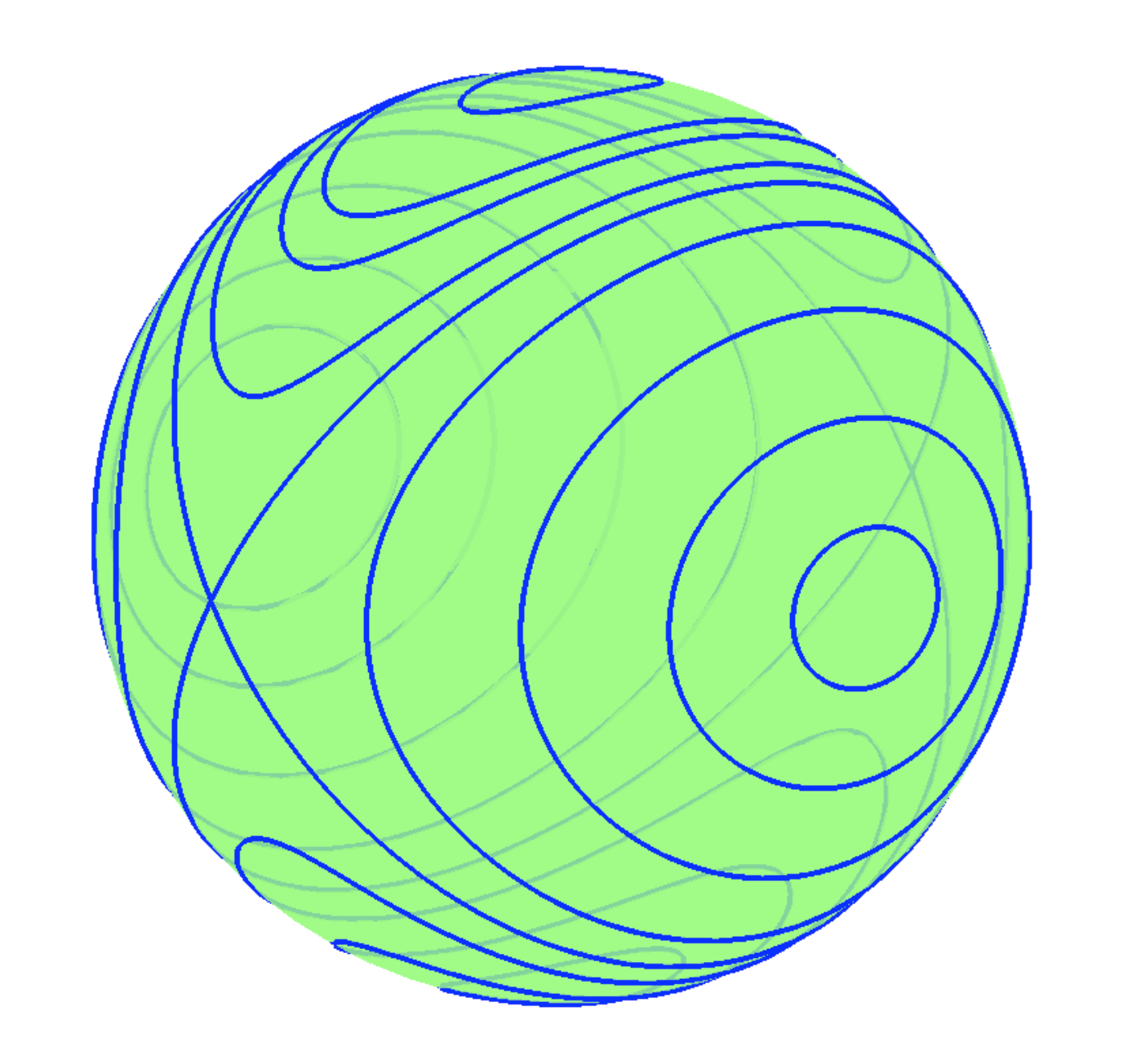}
\caption{\label{fig:so3a} Some orbits of the rigid body Hamiltonian $H = \frac{1}{2}w_1^2 + \frac{1}{4}w_2^2 + \frac{5}{3}w_3^2$ calculated using the 4-dimensional Hopf-fibration realization of $\mathfrak{so}(3)^*$ and the collective midpoint rule and time step 0.04.}
\end{center}
\end{figure}

\begin{figure}
\begin{center}
\includegraphics[width=7cm]{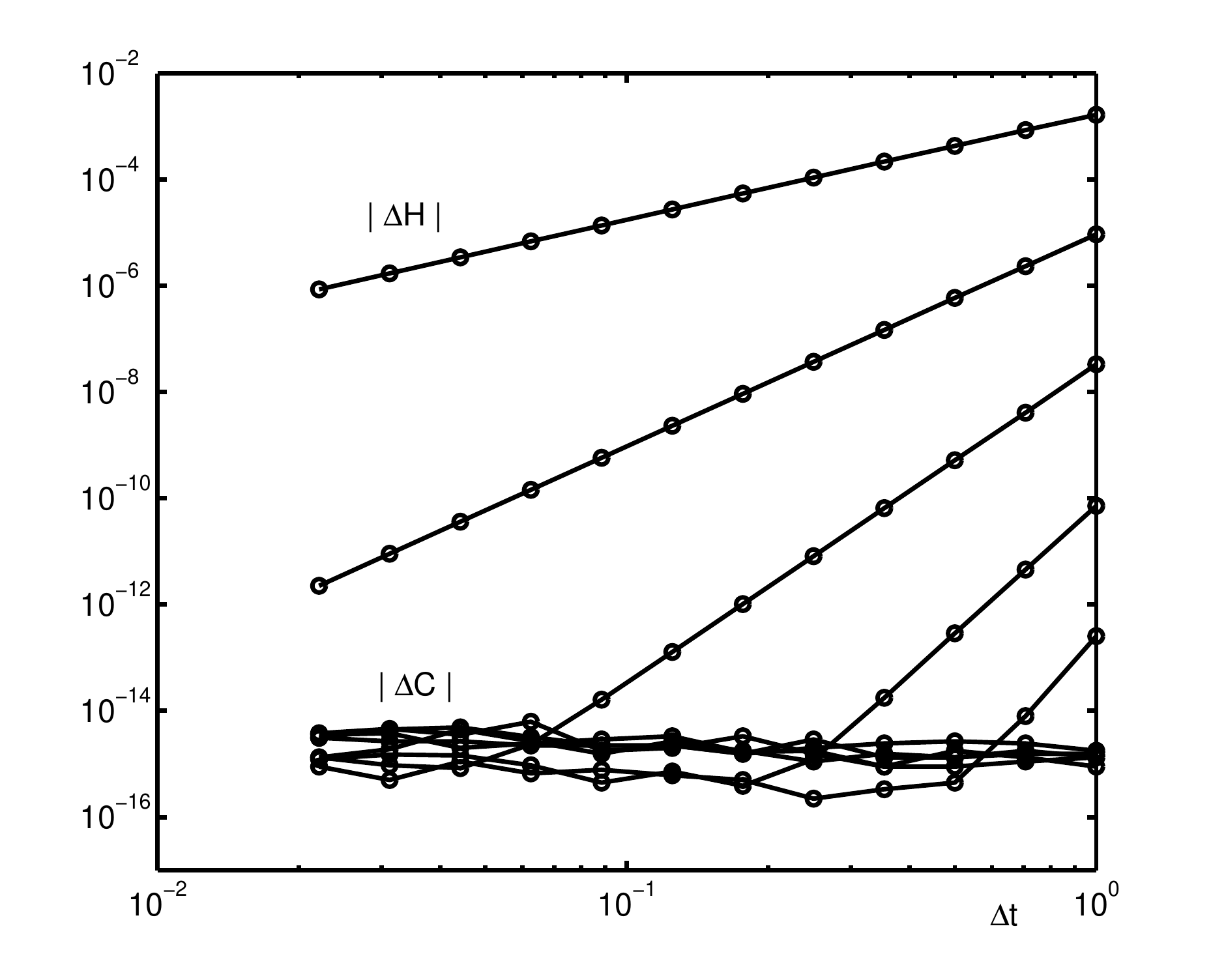}
\includegraphics[width=7cm]{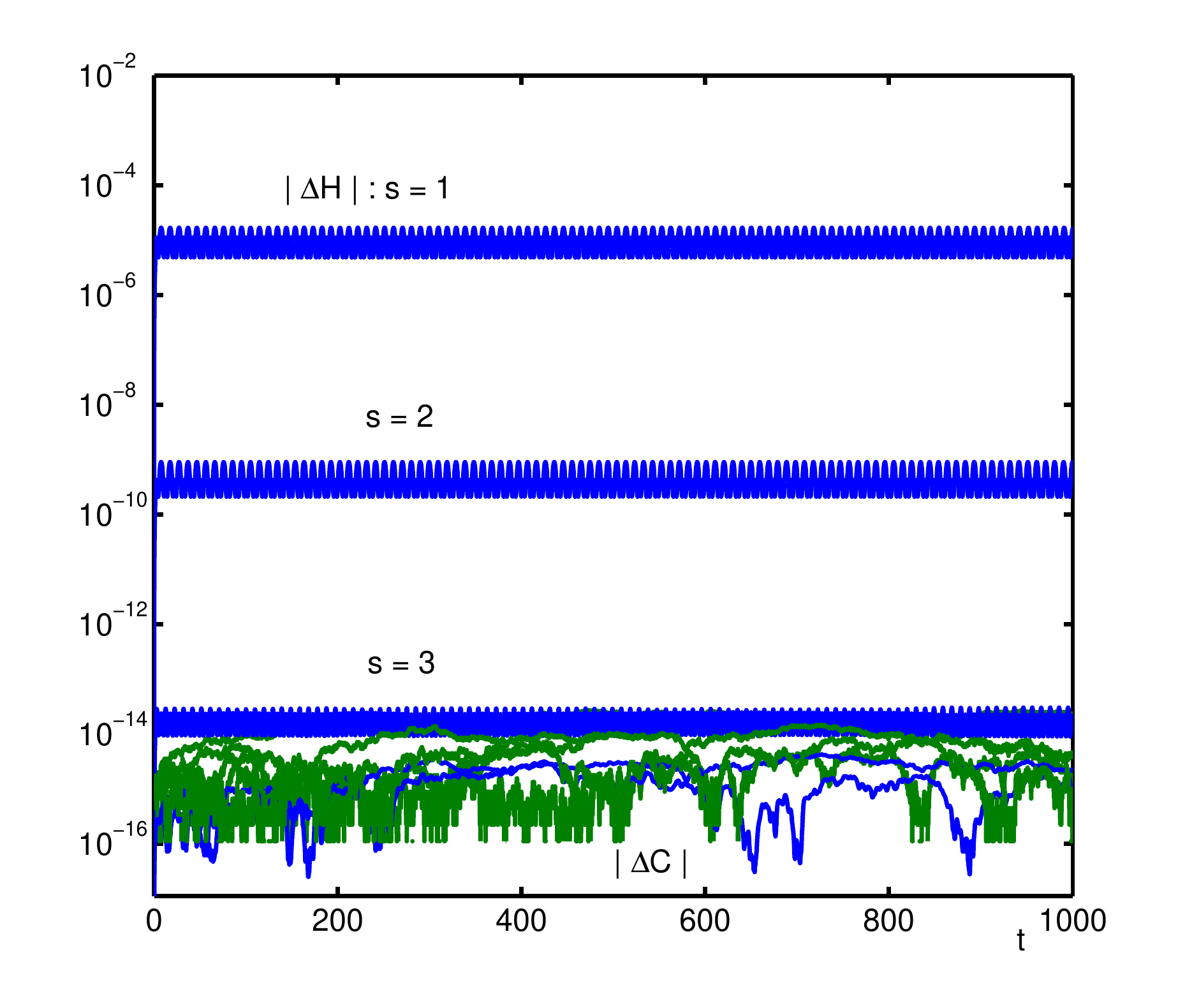}
\caption{\label{fig:so3grk} For the ODE in in Fig. \ref{fig:so3a}, illustrating the behaviour of Gauss Runge--Kutta methods 
with $s=1$--5 stages, orders 2--10. The errors in the Casimir $C=\|w\|^2$ are at roundoff and the errors in the Hamiltonian do not grow. Left: Integration time $t=20$, varying $\Delta t$; Right: $\Delta t = 0.1$, varying $t$.} 
\end{center}
\end{figure}

\begin{figure}
\begin{center}
\includegraphics[width=8cm]{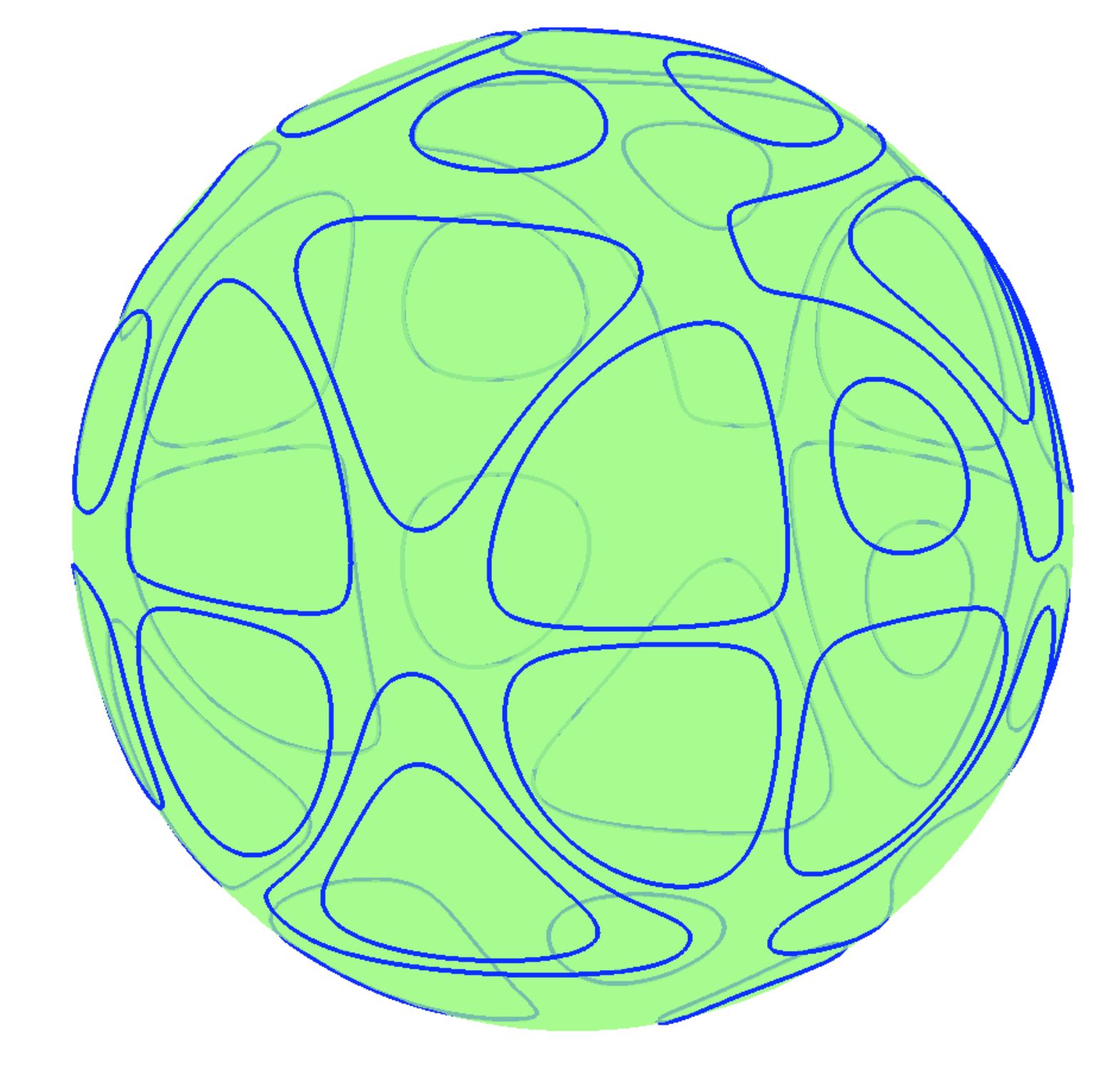}
\caption{\label{fig:so3b} Some orbits of the Hamiltonian $\prod_{i=1}^3\sin 4w_i$ calculated using the 4-dimensional Hopf-fibration realization of $\mathfrak{so}(3)^*$ and the collective midpoint rule with time step 0.01, shown on the coadjoint orbit $\|w\|=1$.}
\end{center}
\end{figure}

\begin{figure}
\begin{center}
\includegraphics[width=8cm]{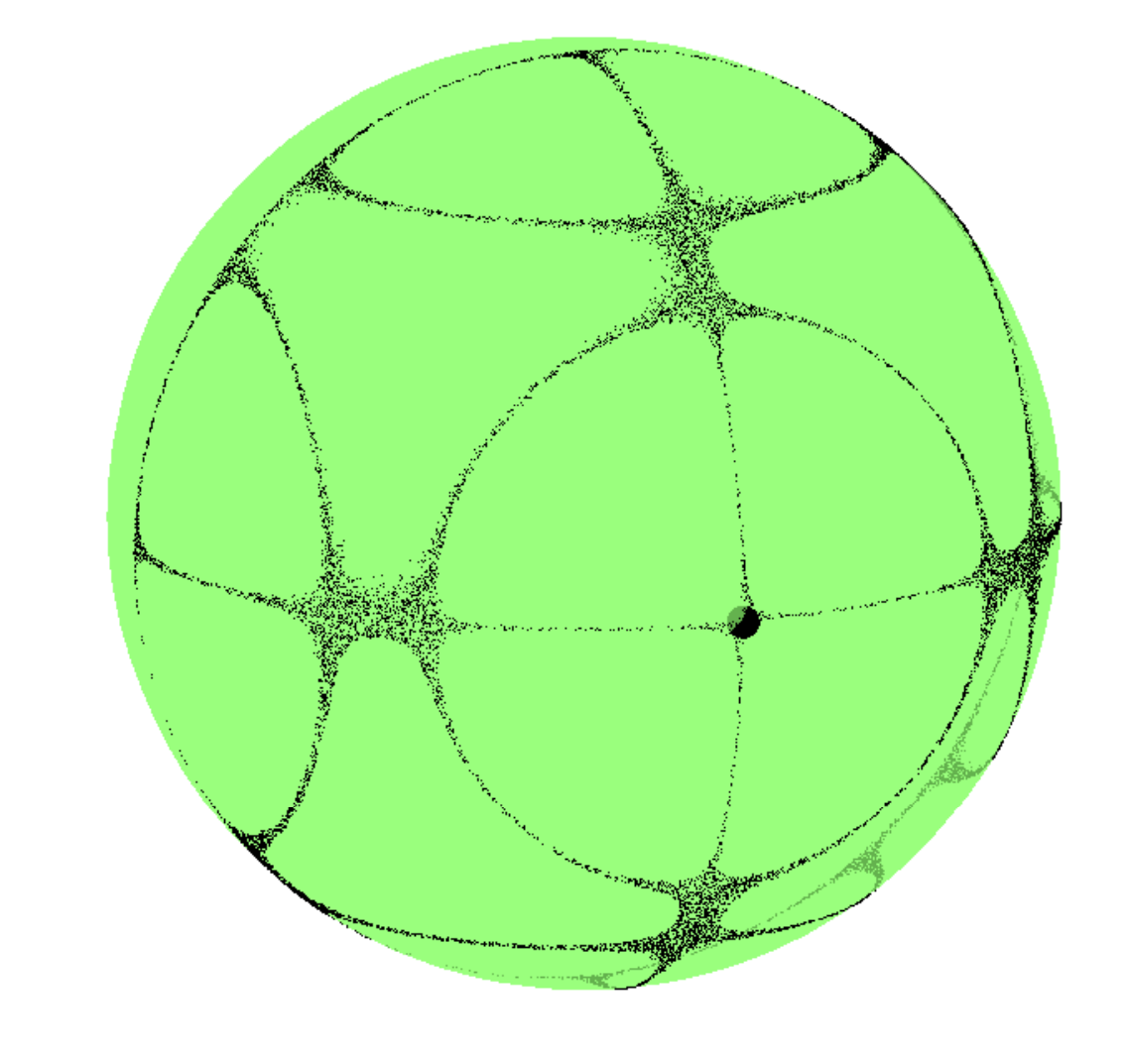}
\caption{\label{fig:so3b} 16000 iterations of the one-period map of the Hamiltonian $\prod_{i=1}^3\sin 4w_i + 0.01 w_1 \sin^2 t$ calculated using the 4-dimensional Hopf-fibration realization of $\mathfrak{so}(3)^*$ and the collective midpoint rule with time step $2\pi/30$, shown on the coadjoint orbit $\|w\|=1$. The marked point is the initial condition $w=(1,0,0)$. The orbit forms a ``chaotic web,'' a typical behaviour in 2-dimensional area-preserving dynamics.}
\end{center}
\end{figure}


\begin{example}{\bf (Hopf fibration realization of $\mathfrak{so}(3)^*$)}
\rm
\label{ex4}
For $\mathfrak{so}(3)^*$, the leaves are the origin (with 3-dimensional isotropy) and the 2-spheres $\|z\| = $const. (with 1-dimensional isotropy) so the minimum possible dimension to cover a sphere is 4. There is a canonical 4-dimensional realization, not just of the neighbourhood of a sphere, but of all of $\mathfrak{so}(3)^*$.
Let $G_1 = SU(2)$ with its natural action on $M={\mathbb C}^2$ which is canonical for $(z_1,z_2)= (q_1 + i p_1,q_2 + i p_2)\in M$. The momentum map $J\colon M\to \mathfrak{su}(2)^*\cong \mathfrak{so}(3)^*\cong {\mathbb C}\times \R\in(w_1+i w_2, w_3)$ is given by $J(z_1,z_2) = (\frac{1}{2} \bar z_1 z_2,\frac{1}{4}(|z_1|^2 - |z_2|^2))$ and is surjective with connected fibres 
(circles and points ${\rm e}^{{\rm i}\alpha}(z_1,z_2)$.) Hamilton's equations for $H\circ J_1$ are
$$ \eqalign{
\dot z_1 &= -\frac{1}{2}\left(i z_2 \frac{\partial H}{\partial w_1} + z_2 \frac{\partial H}{\partial w_2} +  i z_1 \frac{\partial H}{\partial w_3}\right)\cr
\dot z_2 &= \frac{1}{2}\left(- i z_1 \frac{\partial H}{\partial w_1} + z_1 \frac{\partial H}{\partial w_2} +  i z_2 \frac{\partial H}{\partial w_3}\right)\cr
}$$
The only independent invariant of $G_1$ is  $I := |z_1|^2 + |z_2|^2$ and it is quadratic and a complete invariant. Therefore from Theorem \ref{thm:onegroupint}(ii) symplectic Runge--Kutta methods applied to $H\circ J_1$ produce symplectic integrators for $\mathfrak{so}(3)^*$. Some examples are shown in Figs. \ref{fig:so3a}--\ref{fig:so3b}. Another application of these integrators is to generate smooth symplectic integrators for Hamiltonian systems on $S^2$ equipped with the Euclidean area form: extend the Hamiltonian smoothly to $\R^3$ and apply the collective midpoint rule. This gives a symplectic integrator on $S^2$ using 4 variables. Integrators based on canonical two-dimensional charts on $S^2$ use fewer variables, but are not, in general, smooth.
In general one can say that adding extra variables is one of very few fundamental tools available to get methods with new properties: the extra stages of Runge--Kutta methods (that allow, e.g., symplecticity) are an example. \end{example}

In Example \ref{ex4}, $J_1$ admits a complete symmetry group: letting $J_2=I$, $X_{J_2}$ consists of two harmonic oscillators with the same frequency and all orbits (except the origin) are circles, so $G_2 = S^1$. The action of $G_1$ is transitive on the $J_2$-fibres, and vice-versa, so Theorem \ref{thm:transitive} also shows that symplectic Runge--Kutta methods are collective in this example.  However, the full smooth centralizer of $G_1$ is $(\bar z_1z_2, |z_1|^2, |z_2|^2)$, not $J_1$, so $(G_1,G_2)$ do not form a dual pair. The geometry is that of 
the classical Hopf fibration: the $J_2$-fibres ($S^1$) lie in the $G_1$-orbits ($S^3$) giving rise to $S^3/S^1\cong S^2$. This special situation arises because $G_2$ is abelian and so $X_{J_2}$ is $G_2$-invariant. This example is considered in more detail in \cite{smallcollective}.

If the action has discrete isotropy (e.g., if it is free) at $x\in M$ then $J$ is a submersion at $x$ and $\dim J(M)=\dim\g^*$. The product action of $G$ on $M^n$ becomes free on an open subset of $M^n$ for $n$ sufficiently large for most effective group actions \cite{olver}; this will be one of our main tools to construct realizations. One should not take $n$ too large (because that would add too many extra variables) or too small (because that would prevent the action being free). 
This kind of action occurs in the following examples.

Howe \cite{howe} gives a classification of the irreducible reductive dual pairs in $Sp(V)$. 
There are just seven families of these, with $(G_1,G_2) = (GL(n,F),GL(m,F))$ where $F=\R$, $\C$, or  $\mathbb{H}$; $(O(p,q,F),Sp(2k,F))$ (the groups preserving a Hermitian (resp. skew-Hermitian) bilinear form); and $(U(p,q),U(r,s))$. It is straightforward to work out the momentum maps and their range for these dual pairs; we do this for three key dual pairs in the following examples.

\begin{example}\rm\label{ex:OSp}
Let $G_1 = O(n)$, $G_2=Sp(2k)$ and $M=T^*\R^{n\times k}$. We write $X=(Q,P)\in M$ and $\Omega = [0, I; -I, 0]$ for the Poisson structure matrix of $T^*\R^n$. The group actions and momentum maps are
$$\eqalign{
 &G_1\colon\quad A\cdot(Q,P) = (AQ,AP) =AX,\quad\quad J_1(X) = X \Omega X^T =  Q P^T - P Q^T;\cr
 &G_2\colon\quad B\cdot(Q,P) = (QB^T,PB^T) = X B^T,\quad J_2(X) = \Omega X^T X .\cr
 }$$
(Here the skew-symmetric matrix $X\Omega X^T$ pairs with an element of $\mathfrak{o(n)}$ via the standard basis, and the Hamiltonian matrix $\Omega X^T X$ pairs with an element of $\mathfrak{sp}(2k)$ via the standard basis.)
We first consider the range of $J_2$. Clearly $X^T X \in \R^{2k\times 2k}$
is a symmetric positive semidefinite matrix of rank at most $\min(2k,n)$. We claim
that it can be any such matrix. We use a method of classical invariant theory, see \cite{go-wa}, Chapter 5. Let $S$ be a symmetric positive semidefinite
$2k\times 2k$ matrix of rank $n$ where $n\le 2k$. Then $S=L^T D L$ where
$L$ and $D$ are $2k\times 2k$, $L$ is orthogonal, and $D$ is diagonal with diagonal entries $\lambda_1,\dots,\lambda_n,0,\dots,0$ where $\lambda_i>0$ for all $i$.
Define the matrix $X\in\R^{n\times 2k}$ by $X_{ij} = \sqrt{\lambda_i }L_{ij}$
for $i=1,\dots,n$, $j=1,\dots,2k$. Then $X^TX = S$. Thus, $J_2(M)$ is isomorphic
to the space of such matrices. The case $k=1$ was already considered in 
Ex. \ref{ex2}(b) and Figure \ref{fig:sl2}. We have $\dim J_2(M)=\dim \mathfrak{g}_2^*$ when $n\ge 2k$, but (because of the positivity restriction) $J_2$ is never surjective.

A similar argument shows that $J_1(M)$ consists of all antisymmetric $n\times n$ matrices of rank at most $2k$, now with no positivity restriction. Therefore $J_1$ is surjective when $2k\ge n$ (when $n$ is even) or $2k\ge n-1$ (when $n$ is odd).  The cases $n=2k$ and $n=2k+1$ are the most balanced as $J_1(M)$ and $J_2(M)$ are both top-dimensional and $\dim G_1 + \dim G_2 = \dim M$. One of these cases,
$n=3$ and $k=1$, is illustrated in Ex. \ref{ex1}.

This example provides a canonical realization of $\mathfrak{o}(n)^*$ of dimension $2n\lfloor\frac{n}{2}\rfloor$. The lower bound of Weinstein \cite{weinstein} for a realization
of a top-dimensional symplectic leaf is $\frac{1}{2}n(n-1)+\lfloor\frac{n}{2}\rfloor$, because $\dim \mathfrak{o}(n)^* = \frac{1}{2}n(n-1)$ and there are $\lfloor\frac{n}{2}\rfloor$ Casimirs $\tr W^{2i}$, $W\in\mathfrak{o}(n)^*$.

The case when $J(M)$ is a proper, even a low-dimensional subset of $\g^*$ may still be relevant if one wishes to integrate on particular symplectic leaf (as in the central force problem) or lower-dimensional symplectic leaves of $J$. The lowest non-zero dimensional leaves of $\mathfrak{o}(n)^*$, those with $\mathrm{rank}W= 2$, are isomorphic to $O(n)/(U(1)\times O(n-2))$ which has dimension $2n-4$. In this case taking $k=1$ (i.e., taking the canonical action of $O(n)$ on $T^*\R^n$) provides canonical variables for this symplectic manifold.

\end{example}

\begin{example}\rm
Let $G_1 = GL(n)$, $G_2 = GL(k)$, and $M=T^*\R^{n\times k}$.
The group actions and momentum maps are
$$\eqalign{
 & G_1\colon\quad A\cdot(Q,P) = (AQ,A^{-T}P)\,\quad J_1(Q,P) = Q P^T\cr
 & G_2\colon\quad B\cdot(Q,P) = (QB^T,PB^{-1}),\quad J_2(Q,P) = Q^T P\cr
 }$$
A similar calculation as in Ex. \ref{ex:OSp}, but using the singular value decomposition
instead of orthogonal diagonalisation, shows that $J_1(M)$ consists of
all $n\times n$ matrices of rank $\min(n,k)$  (hence $J_1$ is surjective when
$k\ge n$) and $J_2(M)$ consists of
all $k\times k$ matrices rank $\min(n,k)$ (hence $J_2$ is surjective when
$n\ge k$). They are both surjective when $n=k$. This case provides a full realization of $GL(n)$ of dimension $2n^2$, which can be compared to the lower bound
of Weinstein of $n^2 + n$ (there are $n$ Casimirs $\tr W^i$).

\end{example}

\providecommand{\Diff}{\mathrm{Diff}}
\providecommand{\Xcal}{\mathfrak{X}}
\providecommand{\met}{\mathsf{g}}
\providecommand{\LieD}{\mathcal{L}}
\providecommand{\vol}{\mu}
\providecommand{\pair}[1]{\langle #1 \rangle}
\providecommand{\trans}{T}

\begin{example}{\bf (Discretisation of $\Xcal(\R^{d})^*$ by landmarks)} \rm
	In this example we obtain a \emph{partial} realization of an infinite dimensional Lie--Poisson manifold by a finite dimensional symplectic vector space.
	The approach can be seen as a discretisation of the dual of the space of vector fields on $\R^{d}$.

	Consider the infinite dimensional algebra $\mathfrak{g}=\Xcal(\R^d)$ of vector fields on~$\R^d$.
	Formally, this is the Lie algebra of the group $G=\Diff(\R^d)$ if diffeomorphisms of~$\R^d$.
	The bracket on $\Xcal(M)$ is given by $(u,v)\mapsto -\LieD_u v$.
	
	Now, $\Diff(\R^d)$ acts on $q\in Q=\R^{dn}$ by $(q_1,\ldots,q_n)\mapsto \big(\phi(q_1),\ldots,\phi(q_n)\big)$.
	Notice that this is a nonlinear action.
	The corresponding cotangent lifted action on $M=T^*\R^{dn} \simeq \R^{2dn}$ is given by 
	\[
		(q_1,\ldots,q_n,p_1,\ldots,p_n) \mapsto \big(\phi(q_1),\ldots,\phi(q_n), T\phi(q_1)^{-\trans}p_1,\ldots,T\phi(q_n)^{-\trans}p_n\big).
	\]
	Since it is a cotangent lifted action, it is Hamiltonian.
	The corresponding momentum map $J\colon T^*\R^{dn}\to\Xcal(\R^{d})^*$ is given by
	\[
		J(q_1,\ldots,q_n,p_1,\ldots,p_n) = \sum_{i=1}^n p_i \delta(\,\cdot\, - q_i) .
	\]
	Thus, $J$ gives a partial realization of $\Xcal(\R^{d})^*$, which fulfills all the requirements in Theorem \ref{thm:onegroup}.
	
	Now, if $H\colon\Xcal(\R^d)^*\to \R$ is a Hamiltonian, consider the collective Hamiltonian $H\circ J$, which is now on a finite dimensional space symplectic vector space. 
	This example was first obtained in \cite{ho-ma}. 
	It is also described in terms of Clebsch variables in \cite{CoHo2009}.
	
	Integration of $X_{H\circ J}$ is a widely used technique in \emph{computational anatomy}, where the points in $\R^d$ are called landmarks \cite{mu-mi}. An exactly analogous example is already present in \cite{ma-we} in which $\Diff_{\rm vol}(\R^2)$ acts on the coadjoint orbit in $\Xcal^*_{\rm vol}(\R^2)$ consisting of point vortices to give the usual canonical description of point vortices.
\end{example}

\begin{example}\rm
\label{ex:bifol}
Some of the previous examples can be interpreted as instances of Theorem \ref{thm:isotropic}. For example, Example \ref{ex2}(b) can be constructed  by starting with $M=T^*\R^2$, $k=1$, and $f_1 = q\cdot p$.
The Hamiltonian vector field $X_{f_1}$ has 4 quadratic first integrals, namely $p_i q_j$ for $i,j=1,2$, so the midpoint
rule applied to $H(p_1 q_1, p_1 q_2, p_2 q_1, p_2 q_2)$ preserves $q\cdot p$ and descends to $M/X_{f_1}$. This approach does not identify the Poisson manifold. Using the independent invariants $p_1q_2$, $p_2 q_1$, and $p_1 q_1 - p_2 q_2$, i.e., $J$ in Example \ref{ex2}(b), identifies $M/X_{f_1}$ as the Lie--Poisson manifold $\mathfrak{sl}(2)^*$. Alternatively, it is possible
to extend $f_1$ by commuting functions so as to form a surjection $(f_1,f_2,f_3)$ whose fibres are the orbits of $X_{f_1}$---$(q\cdot p, p_1 q_2, p_2 q_1)$ will do---but again this does not identify the Poisson manifold. In this case
$\{f_2,f_3\} = \sqrt{f_1^2 - 4 f_2 f_3}$ so it is not obvious that we have contructed the Lie--Poisson manifold of $\mathfrak{sl}(2)^*$.

Similarly, the Hopf fibration (Example \ref{ex4}) can be constructed by starting with $M=T^*\R^2$, $k=1$, and $f_1 = \|q\|^2+\|p\|^2$.

Note that although $f_1,\dots,f_k$ must be quadratic, the invariants of $X_{f_i}$ need not be quadratic, and thus the foliations need not arise from a Howe dual pair. An example arises in $T^*\R^2$ with $k=1$ and $f_1 = a_1 q_1 p_1 + a_2 q_2 p_2$. The only quadratic invariants of $X_{f_1}$ (for generic $a_i$) are $q_1 p_1$ and $q_2 p_2$, but the orbits of $X_{f_1}$ are 1-dimensional, and the third invariant $q_1^{1/a_1} q_2^{-1/a_2}$ is not quadratic.

\end{example}

\section{Discussion}

The integrators presented here are undoubtedly the simplest possible symplectic integrators for general Hamiltonians on Lie-Poisson manifolds. The method is uniform in $H$, but not in $\mathfrak{g}^*$. 
This prompts several questions: Can $(M,J)$ with $J(M)=\mathfrak{g}^*$ be constructed algorithmically from $\mathfrak{g}^*$? Can it be done canonically? What is the minimum dimension of $M$? The second key requirement of the method is that the group orbits be fibres of quadratic functions (or their intersection with invariant polyhedral sets), so we can ask the same questions under this restriction.

\ack

O.~Verdier would like to acknowledge the support of the \href{http://wiki.math.ntnu.no/genuin}{GeNuIn Project}, funded by the \href{http://www.forskningsradet.no/}{Research Council of Norway},  as well as the hospitality of the \href{http://ifs.massey.ac.nz/}{Institute for Fundamental Sciences} of Massey University, New Zealand, where some of this research was conducted.
K.~Modin would like to thank the \href{http://www.math.ntnu.no}{Department of Mathematics at NTNU} in Trondheim and the \href{http://www.vr.se}{Swedish Research Council} for support. This research was supported by the 
 \href{http://wiki.math.ntnu.no/crisp}{CRiSP Project}, funded by the \href{http://cordis.europa.eu/fp7/home_en.html}{European Commission's Seventh Framework Programme}, and by the Marsden Fund of the Royal Society of New Zealand.



\end{document}